\documentclass[11pt]{article}

\usepackage{amssymb, enumerate}
\usepackage{amsfonts}
\usepackage{amsmath}

\newcommand{\R}{\mathbb R}
\newcommand{\C}{\mathbb{C}}
\newcommand{\Z}{\mathbb{Z}}
\newcommand{\N}{\mathbb{N}}
\newcommand{\Q}{\mathbb{Q}}
\newcommand{\A}{\mathbb{A}}
\newcommand{\K}{\mathbb{K}}

\newcommand{\cA}{\mathcal{A}}
\newcommand{\cO}{\mathcal{O}}
\newcommand{\cU}{\mathcal{U}}

\newtheorem{proposition}{Proposition}
\newtheorem{theorem}[proposition]{Theorem}

\newtheorem{remark}[proposition]{Remark}
\newtheorem{corollary}[proposition]{Corollary}
\newtheorem{lemma}[proposition]{Lemma}

\newenvironment{proof}{
\trivlist \item[\hskip \labelsep\mbox{\it Proof:
}]}{\hfill\mbox{$\square$}
\endtrivlist}

\title{Elimination for generic sparse polynomial systems\footnote{Partially supported by the Argentinian research grants
CONICET PIP 0099/11, UBACYT 20020090100069 (2010-2013) and UBACYT 20020120100133 (2013-2016).}}
\author{Mar\'\i a Isabel Herrero$^{\sharp}$, Gabriela Jeronimo$^{\sharp,\dag,\diamond}$, Juan
Sabia$^{\dag,\diamond}$\\[5mm]
{\small $\sharp$ Departamento de Matem\'atica, Facultad de Ciencias
Exactas y
Naturales,} \\[-1mm] {\small Universidad de Buenos Aires, Ciudad
Universitaria, (1428) Buenos Aires, Argentina}\\[2mm]
{\small $\dag$ Departamento de Ciencias Exactas, Ciclo B\'asico
Com\'un,}\\[-1mm]
{\small Universidad de Buenos Aires, Ciudad Universitaria, (1428)
Buenos Aires, Argentina}\\[2mm]
{\small $\diamond$ IMAS, CONICET, Argentina }}

\begin{document}

\maketitle

\begin{abstract}
We present a new probabilistic symbolic algorithm that, given a variety defined in an $n$-dimensional affine space by a generic sparse system with fixed supports, computes the Zariski closure of its projection to an $\ell$-dimensional coordinate affine space with $\ell<n$. The complexity of the algorithm depends polynomially on combinatorial invariants associated to the supports.
\end{abstract}

\section{Introduction}

Let $\cA=(\cA_1,\dots, \cA_r)$ be a family of finite subsets of $(\Z_{\ge 0})^n$ and $\mathbf{f}=(f_1,\dots, f_r)$ a system of polynomials in $\Q[X_1,\dots, X_n]$ supported on $\cA$. If $V(\mathbf{f}) \subset \C^n$ denotes the affine variety of the common zeros of the polynomials in $\mathbf{f}$ and, for a given $\ell<n$, $\pi: \C^n \to \C^\ell$ is the projection $\pi(x_1,\dots, x_n) = (x_1,\dots, x_\ell)$, we consider the problem of computing algorithmically a description of the Zariski closure $\overline{\pi(V(\mathbf{f}))}\subset \C^\ell$ within a complexity depending on combinatorial invariants associated to the input supports.

The computation of (Zariski closures of) linear projections of varieties is the basic task in elimination theory. A more general formulation of this problem is algorithmic quantifier elimination over algebraically closed fields (see, for instance, \cite{Tarski51}, \cite{Hei83}, \cite{CG84}, \cite{FGM90}, \cite{PS98} for algorithms with complexities depending on the number and degrees of the polynomials and the number of variables involved). Particular instances of the computation of Zariski closures of projections are the computation of Chow forms (see, for instance, \cite{Caniglia90}, \cite{GH91}, \cite{JKSS04}),  classical resultants (see \cite{DD01} and the references therein) and sparse resultants (see, for example, \cite{Stu94}, \cite{CE00}, \cite{DAndrea02}, \cite{JS07}).

The foundations of the study of sparse polynomial systems can be traced back to \cite{Ber75}, \cite{Kho78} and \cite{Kus76}, which prove that the number of isolated roots in $(\C^*)^n$ of a system of $n$ polynomial equations in $n$ unknowns is bounded by the mixed volume of the family of their supports.

This result led to the construction of algorithms for the computation of these isolated solutions, particularly designed to deal with sparse systems (see, for instance, \cite{VVC94}, \cite{HS95}, \cite{Roj03}, \cite{JMSW09}), which run faster than the known general procedures solving this task.
Later on, upper bounds for the number of isolated \emph{affine} solutions of sparse systems were obtained and new efficient algorithms to compute them were designed (see, for instance, \cite{Roj94}, \cite{LW96}, \cite{RW96}, \cite{HS97}, \cite{EV99}, \cite{GLW99}, \cite{HJS10}).
More recently, positive dimensional components of affine varieties defined by sparse systems were also considered: in \cite{Ver09}, certificates for the existence of $1$-dimensional components were given and, in \cite{AV12} and \cite{AV13}, under certain assumptions on the equations, algorithmic methods to describe Puiseux series expansions of curves and arbitrary positive dimensional components, respectively, were presented.
Also, in \cite{HJS}, an upper bound in terms of mixed volumes for the degree of the affine variety defined by a sparse polynomial system of $n$ equations in $n$ unknowns was proved and algorithms for characterizing the equidimensional decomposition of affine varieties defined by sparse systems were designed.
The good performance of the most efficient algorithms dealing with sparse systems relies on the use of polyhedral deformations introduced in \cite{HS95}, because these deformations essentially preserve the monomial structure of the system involved.

In this paper, we present a probabilistic symbolic algorithm which computes  the Zariski closure $\overline{\pi(V(\mathbf{f}))}\subset \C^\ell$ for a sparse system $\mathbf{f}$ with fixed supports and \emph{generic} coefficients. We use the decomposition of the variety $V(\mathbf{f})$ into equidimensional subvarieties contained in coordinate subspaces proved in \cite{HJS} to reduce the problem to the case of a variety such that each of its components intersects the torus. In this case we compute a \emph{geometric resolution} of the variety with respect to a suitably chosen set of free variables and, from this resolution, we show how to obtain a geometric resolution of the Zariski closure of the required projection. The complexity of our algorithm is polynomial in combinatorial invariants associated to the supports of the input polynomials. Our main result is the following:

\begin{theorem}\label{thm:Q-Proj}
Let $\mathbf{f}=(f_1,\dots, f_r)$ be a system of polynomials in $\Q[X_1,\dots, X_n]$ with generic coefficients supported on a family $\cA= (\cA_1,\dots, \cA_r)$ of finite
subsets of $(\mathbb{Z}_{\ge 0})^n$  such that $\dim (\sum_{j\in J} \cA_j )\ge \# J$ for every $J\subset \{1,\dots, r\}$, and let $V^*(\mathbf{f})\subset \C^n$ be the Zariski closure of  $\{ x\in (\C^*)^n \mid f_j(x) = 0 \hbox{ for every } 1\le j \le r\}$. There is a probabilistic algorithm that computes a geometric resolution of $\overline{\pi(V^*(\mathbf{f}))}$, where $\pi: \C^n \to \C^\ell$ is the projection to the first $\ell$ coordinates, within
$O( r \mathcal{D}( n^2  N \log(d)  (\mathcal{D}^2 + \mathcal{E})  +\mathcal{D}^{5} ))$
 operations in $\Q$ up to logarithmic factors, where  $\mathcal{D}=MV(\cA, \Delta^{(n-r)})$, $N= \sum_{1\le j \le r} \# \cA_j$, $d:= \max_j\{ \deg(f_j) \}$ and $\mathcal{E}= \sum_{1\le h \le r} MV( ((\cA_j)_{j\ne h}), \Delta^{(n-r+1)})$. Here, $MV$ denotes mixed volume, $\Delta$ is the vertex set of the standard simplex in $\R^n$ and the superscript indicates the number of repetitions.
\end{theorem}

The stated complexity is due to a polyhedral deformation based algorithm to solve generic sparse zero-dimensional systems (\cite{JMSW09}), an algorithmic Newton-Hensel lifting (\cite{Schost03}) and codification of multivariate polynomials as straight-line programs (\cite{BCS97}). As in \cite{JMSW09}, we do not take into account the cost of mixed volume related computations (see Section \ref{sec:algcod} for more details on this point). Furthermore, we have ignored terms depending on the size of certain combinatorial objects associated to polyhedral deformations involved; for a more precise complexity estimate, see Theorem \ref{thm:K-Proj}.

The paper is organized as follows. In Section \ref{sec:prelim}, the notation and the basic theoretical and algorithmic notions  used throughout are introduced. Section \ref{sec:theoretical} is devoted to proving the main theoretical results on which our algorithms rely. Finally, Section \ref{sec:algorithms} contains the descriptions, proof of correctness and complexity estimates of our algorithms and examples illustrating how they work.

\section{Preliminaries}\label{sec:prelim}

\subsection{Basic definitions and notation}

Let $k$ be a field of characteristic zero and $\overline{k}$ be an algebraic
closure of $k$. Given polynomials $f_1,\dots, f_r \in k[X_1,\dots,X_n]$, we write
$V(\mathbf{f})=V(f_1,\dots, f_r)$ for the affine variety which is the set of the common zeros of
$\mathbf{f}=(f_1,\dots, f_r)$ in $\overline k^n$, and
$$V^*(\mathbf{f})= V^*(f_1,\dots, f_r):= \overline{V(f_1,\dots, f_r) \cap (\overline k^*)^n}$$
for the union of the irreducible components of $V(f_1,\dots, f_r)$ containing points with all their coordinates in $\overline k^*:= \overline k - \{0\}$

For a variety $V\subset \overline k^n$ definable over $k$, we denote $k[V]=k[X_1,\dots, X_n]/I(V)$ its coordinate ring  (where $I(V)\subset k[X_1,\dots, X_n]$ is the ideal of the polynomials vanishing identically on $V$). If $V$ is irreducible, we write $k(V)$ for the fraction field of $k[V]$.

To describe zero-dimensional affine varieties we use the notion of a \emph{geometric resolution}
(see, for instance, \cite{GLS01} and the references therein):
Let $V= \{\xi^{(1)}, \dots, \xi^{(D)}\}\subset \overline k^n$ be a
zero-dimensional variety defined by polynomials in $k[X_1,\dots, X_n]$. Given a
linear form $\lambda = \lambda_1X_1 + \dots + \lambda_n X_n$ in $k[X_1,
\dots, X_n]$ such that $\lambda(\xi^{(i)})\ne \lambda(\xi^{(j)})$ if $i
\ne j$, the following polynomials completely characterize $V$:
\begin{itemize}
\item the minimal polynomial $q_\lambda = \prod_{1 \le i \le D}(Y - \lambda(\xi^{(i)})) \in k[Y]$ of $\lambda$ over
the variety $V$ (where $Y$ is a new variable),
\item polynomials $v_1, \dots, v_n \in k[Y]$ with $\deg(v_j) < D$ for every $1 \le j \le n$
satisfying $V = \{(v_1 (\eta), \dots, v_n (\eta)) \in \C^n \mid \eta
\in \overline k, q_\lambda(\eta)=0\}$.
\end{itemize}
The family of univariate polynomials $(q_\lambda, v_1,\dots, v_n) \in
k[Y]^{n+1}$ is called the geometric resolution of $V$ (or a geometric resolution of $k[V]$) associated
with the linear form $\lambda$.

The notion of geometric resolution can be extended to any equidimensional variety:
Let $V\subset \overline k^{n}$ be an equidimensional variety of dimension $t$
defined by polynomials in $k[X_1,\dots, X_n]$. Assume that, for each irreducible component $W$ of $V$, the identity $I(W) \cap k[X_1,\dots, X_t] = \{ 0 \}$ holds. By considering $k(X_1,\dots, X_t) \otimes k[V]$, we are in a zero-dimensional situation, and we call a \emph{geometric resolution of $V$ with free variables $X_1,\dots, X_t$} to a geometric resolution $(q_\lambda, v_{t+1},\dots, v_n) \in k(X_1,\dots, X_t)[Y]^{n-t+1}$ of $k(X_1,\dots, X_t) \otimes k[V]$ associated to a linear form $\lambda \in k[X_{t+1},\dots, X_n]$. If $\hat q_\lambda \in k[X_1,\dots, X_t,Y]$ is  obtained from $q_\lambda$ by clearing denominators, a geometric resolution of $V$ gives a birational map between the hypersurface $\{(x_1,\dots, x_t, y) \in \overline k^{n-t+1} \mid \hat q_\lambda(x_1,\dots, x_t, y) =0\}$ and $V$.

When dealing with varieties defined by sparse polynomial systems, an important combinatorial invariant associated to the system is the mixed volume of their supports. For a family $\cA=(\cA_1,\dots, \cA_n)$ of $n$ finite subsets of $(\Z_{\ge 0})^n$, $MV_n(\cA)$ (or $MV(\cA)$ if $n$ is clear from the context) denotes the $n$-dimensional mixed volume of the convex hulls in $\R^n$ of $\cA_1,\dots, \cA_n$ (see, for instance, \cite[Chapter 7]{CLO98} for the definition and basic properties). In this context, we write $\Delta$ for the vertex set $\{0, e_1,\dots, e_n\}$ of the standard simplex of $\R^n$, which is the support of an affine linear form with nonzero coefficients, and $\Delta^{(t)}$ for the family of $t$ copies of $\Delta$.

\subsection{Algorithms and codification} \label{sec:algcod}

Although we work with polynomials, our algorithms only deal with
elements in a base field $k$. The notion of \emph{complexity} of an algorithm we
consider is the number of operations and comparisons in $k$ it performs. We will encode multivariate polynomials in different
ways:
\begin{itemize}
\item in sparse form, that is, by means of the list of pairs $(a, c_a)$ where $a$ runs
over a fixed set including the exponents of the monomials appearing in the
polynomial with nonzero coefficients and $c_a$ is the corresponding
coefficient,
\item in the standard dense form, which encodes a polynomial of degree bounded by $d$
as the vector of the coefficients of all the monomials of degree at most $d$ including zeros (we use this
encoding only for univariate polynomials),
\item  in the \emph{straight-line program} (slp for short) encoding. A straight-line
program is an algorithm without branchings which allows the
evaluation of the polynomial at a generic value (for a precise
definition and properties of slp's, see \cite{BCS97}).
\end{itemize}

In our complexity estimates, we use the usual $O$ notation: for
$f,g\colon \Z_{\ge 0} \to \R$,  $f(d) = O(g(d))$ if $|f (d)| \le c|g(d)|$
for a positive constant $c$. We also use the notation $M(d) = d
\log^2(d) \log(\log(d))$, where $\log$ denotes logarithm to base
$2$. We recall that multipoint evaluation and interpolation of
univariate polynomials of degree $d$  with coefficients in a
commutative ring $R$ of characteristic $0$ can be performed with
$O(M(d))$ operations and that multiplication and division with
remainder of such polynomials can be done with $O(M(d)/ \log(d))$
arithmetic operations in $R$.
We write $\Omega$ for the exponent ($\Omega <4$) in the complexity
$O(d^{\Omega})$ of operations (determinant and adjoint computations) on $d\times d$ matrices
with entries in a commutative ring $R$ (\cite{Ber84}). We use Pad\'e approximation in order to compute the dense representation
of the numerator and denominator of a rational function $f = \frac{p}{q}\in k(Y)$ with
$\max\{\deg p, \deg q\} \le d$ from its Taylor series expansion up to order $2d$, which we do by means of subresultant computations within $O(d^{\Omega +1})$ arithmetic operations in $k$ (see \cite[Corollaries 5.21 and 6.48]{vzGG99}).

Our algorithms are probabilistic in the sense that they make random
choices of points (that we consider cost-free in our complexity estimates) which lead to a correct computation provided the
points lie outside  certain proper Zariski closed sets of  suitable
affine spaces. Using the Schwartz-Zippel lemma (\cite{Schwartz80, Zippel93}),
the error probability of our algorithms
can be controlled by making these random choices within sufficiently
large sets of integer numbers whose size depend on the degrees of the polynomials
defining the previously mentioned Zariski closed sets.

We use two previous algorithms as subroutines:
\begin{itemize}
\item a probabilistic algorithm that, given $n$ generic sparse polynomials in $n$ variables with coefficients in $k$, computes a geometric resolution of the set of their common zeros in $\overline k^n$ (see \cite[Section 5]{JMSW09}),
\item a Newton-Hensel based procedure that, given a system $\mathbf{f}$ of $n$ polynomials in $n$ variables and $t$ parameters, a specialization point $\xi\in \Q^t$ for the parameters and a geometric resolution of the set of simple common zeros of $\mathbf{f}(\xi, \cdot)$, computes an approximation up to a given precision of the geometric resolution of the components of $V(\mathbf{f})$ where the Jacobian determinant of the system with respect to its $n$ variables does not vanish identically (see \cite[Section 4.2]{Schost03}).
\end{itemize}

The algorithm in \cite{JMSW09} assumes that the mixed cells in a fine mixed subdivision of the family of the input supports are given (see \cite{HS95} for the definition of a fine mixed subdivision), that is to say, the computation of mixed cells is considered as a pre-processing. In this paper, our algorithms also require deciding whether a mixed volume is zero or not, and computing mixed volumes.

While the non-vanishing of a mixed volume can be decided algorithmically in polynomial time, the problem of computing mixed volumes is known to be \#P-hard (see \cite{DGH98}, \cite{EC95}).
 The mixed cells in a fine mixed subdivision of a family of finite sets of $\Z^n$ (and, therefore, the mixed volume) can be obtained algorithmically by means of \emph{lifting} (see \cite{HS95}) and linear programming based procedures. In \cite{EC95}, an algorithm following this approach is presented
with a worst-case complexity single exponential in $n$. Successive algorithms that run faster according to numerical results can be found in \cite{LL01} and \cite{GL03}.
A dynamic approach which does not use a random lifting function is given in \cite{VGC96}.
The dynamic enumeration procedure from \cite{MTK07}, which proved to be efficient even for large systems, seems to be the fastest known up until now.
However, there are no explicit complexity upper bounds for these more efficient procedures, neither depending on the input nor the output size (namely, the mixed volume).

As in \cite{JMSW09}, we consider the computation of mixed cells as a pre-processing and do not include its cost in our complexity estimates. Note that if our algorithm needs to be applied to several systems with the same supports, this pre-processing only has to be carried out once.

\section{Theoretical results}\label{sec:theoretical}

Let $n, r$ be positive integers and $\cA= (\cA_1,\dots, \cA_r)$ a family of $r$ finite subsets of $(\Z_{\ge 0})^n$.
Consider $r$ polynomials in $n$ variables $X=(X_1,\dots, X_n)$ supported on $\cA$ with indeterminate coefficients:
for every $1\le j \le r$, let
\begin{equation}\label{eq:genericpols}
f_j(C_j, X)= \sum_{a\in \cA_j} C_{j, a} X^a,
\end{equation}
where
$C_j:=(C_{j,a})_{a\in \cA_j}$ is a set of $N_j:= \# \cA_j$ new indeterminates. Let $\mathbf{f}:=(f_1,\dots, f_r)$ and $K:= \Q(C_1,\dots, C_r)$.

\subsection{Reduction to the toric case}

Consider the family $\mathbf{f}=(f_1,\dots, f_r)$ of polynomials in $K[X_1,\dots, X_n]$ supported on $\cA$ with indeterminate coefficients introduced in (\ref{eq:genericpols}). Since $\mathbf{f}$ is a generic system supported on $\cA$, the equidimensional decomposition of the affine variety $V(\mathbf{f})\subset \overline K^n$ depends only on the combinatorial structure of $\cA$ (see \cite{HJS}). Moreover, its equidimensional components can be defined from certain smaller polynomial families $\mathbf{f}_I$ in fewer variables associated to $\mathbf{f}$. This decomposition, which we explain below, enables us to reduce the problem of computing a projection of the affine variety $V(\mathbf{f})$ to the computation of finitely many projections of varieties having a non-empty intersection with the torus.

\bigskip

For $I\subset \{ 1\dots, n\}$, given a polynomial $f\in K[X_1,\dots, X_n]$, we write $f_I$ to denote the polynomial in $K[(X_i)_{i\notin I}]$ obtained from $f$ by specializing $X_i=0$ for every $i\in I$. We define
 $
 J_I = \{ j\in \{1,\dots, r\} \mid \exists a\in \cA_j : a_i =0 \ \forall\, i\in I \}$  and $
\mathbf{f}_I= ((f_j)_I)_{j\in J_I}$
(note that $J_I$ is the set of indices of those polynomials $f_j$ that do not vanish identically when specializing $X_i=0$ for every $i\in I$). Let $$\Gamma=\{ I \subset \{1,\dots, n\} \mid \forall\, J\subset J_I,\ \dim(\sum_{j\in J} \cA_j^I ) \ge \#J; \forall\, \widetilde I \subset I, \ \#J_{\widetilde I} + \# \widetilde I \ge \#J_{I} + \# I\},$$
where, for every $1\le j \le r$,  $\cA_j^I\subset (\Z_{\ge 0})^{n-\#I}$ denotes the support of $(f_j)_I$. Finally, let $\varphi_I: \overline K^{n-\#I} \to \overline K^n$ be the map that inserts zeros in the coordinates indexed by $I$.
Using the previous notation, according to \cite[Theorem 7]{HJS}, we have that
$$V(\mathbf{f}) = \bigcup_{I\in \Gamma} \varphi_I(V^*(\mathbf{f}_I)).$$

Let $\pi : \overline K^n \to \overline K^\ell$ be the projection $\pi(x_1,\dots, x_n)= (x_1,\dots, x_\ell)$. For  $I\subset \{1,\dots, n\}$, let
$\ell_I =\ell - \# (I \cap \{1,\dots, \ell\}),$
 $\pi_{\ell_I} : \overline K^{n-\# I}\to \overline K^{\ell_I}$ the projection to the first coordinates and $\varphi_{I\cap\{1,\dots,\ell\}}: \overline K^{\ell_I} \to \overline K^\ell$ the map that inserts zeros in the coordinates indexed by $I\cap \{1,\dots,\ell\}$. As
$\pi(\varphi_I(V^*(\mathbf{f}_I))) = \varphi_{I\cap\{1,\dots,\ell\}}(\pi_{\ell_I}(V^*(\mathbf{f}_I)))$
holds for every $I$, we conclude that
$$\pi(V(\mathbf{f})) = \bigcup_{I\in \Gamma} \varphi_{I\cap\{1,\dots,\ell\}}(\pi_{\ell_I}(V^*(\mathbf{f}_I))).$$

Therefore, the computation of $\overline{\pi(V(\mathbf{f}))}$ amounts to obtaining the Zariski closures of the projections by $\pi_{\ell_I}$ of the affine varieties $ V^*(\mathbf{f}_I)= \overline{\{ x\in (\overline K^*)^{n-\# I} \mid \mathbf{f}_I(x)=0 \}}$. For this reason, in the sequel we will deal with this problem.

\subsection{The vanishing ideal of the projection}

We consider a generic polynomial system $\mathbf{f}:=(f_1,\dots, f_r)$ supported on $\cA:=(\cA_1,\dots, \cA_r)$  with indeterminate coefficients $C_j:=(C_{j,a})_{a\in \cA_j}$ for every $1\le j\le r$ as in equation (\ref{eq:genericpols}). We assume that $r\le n$ and that, for every $J\subset \{1,\dots, r\}$, $\dim(\sum_{j\in J} \cA_j)\ge \# J$ (or, equivalently, $MV(\cA, \Delta^{(n-r)})>0$) so that $V^*(\mathbf{f})$ is not empty.

Our aim is to obtain a geometric resolution of the Zariski closure $\overline{\pi(V^*(\mathbf{f}))}$, where $\pi: \overline K^n\to \overline K^\ell$ ($\ell\le n$) is the projection defined by $\pi(x_1,\dots, x_n)= (x_1,\dots, x_\ell)$.
We begin by proving some basic results on the vanishing ideals of the varieties $V^*(\mathbf{f})$ and $\overline{\pi(V^*(\mathbf{f}))}$.

\begin{lemma} \label{lem:idealV*} Under the previous assumptions,
 $(f_1,\dots, f_r): (X_1\dots X_n)^{\infty}$ is a prime ideal in $K[X_1,\dots, X_n]$ of dimension $n-r$. Moreover,
$I(V^*(\mathbf{f}))= (f_1,\dots, f_r): (X_1\dots X_n)^\infty.$
\end{lemma}

\begin{proof} For every $1\le j \le r$, fix $a_{j0}\in \cA_j$ and let $\cA_j':= \cA_j-\{a_{j0}\}$ and $C'_j:=C_j -\{C_{j, a_{j0}}\}$. We denote $C:=(C_1,\dots, C_r)$ and $C':=(C_1',\dots, C_r')$. Consider the ring morphism $\psi: \Q[C,X] \to \Q[C',X]_{X_1\dots X_n}$ (where $\Q[C',X]_{X_1\dots X_n}$ is the localization of $\Q[C',X]$ at the multiplicative set $\{(X_1\dots X_n)^m \mid m\in \Z_{\ge 0}\}$) defined by
$\psi(X_i) =X_i$ for $1\le i \le n$, $\psi(C_{j,a})=C_{j,a}$ for $1\le j \le r$, $a\in \cA'_j$, and $\psi(C_{j,a_{j0}}) = -X^{-a_{j0}}(\sum\limits_{a\in \cA_j'} C_{j,a}X^{a})$ for $1\le j \le r$.

We claim that $\ker(\psi) = (f_1,\dots, f_r):(X_1\dots X_n)^\infty$. It is clear that for a polynomial $g\in(f_1,\dots, f_r):(X_1\dots X_n)^\infty$, $\psi(g)=0$. Assume now that for a polynomial $g\in \Q[C,X]$ we have that $\psi(g) = 0$. Let $C_0:=(C_{1,a_{10}},\dots, C_{r,a_{r0}})$ and let $\widehat C_0:=(\widehat C_{10},\dots, \widehat C_{r0})$ be new variables. By Taylor expansion,
$g(\widehat C_0,C',X) = g(C_0, C',X) + \sum_{1\le j \le r} (\widehat C_{j0}-C_{j,a_{j0}}) \cdot G_j $
for certain polynomials $G_j\in \Q[\widehat C_0, C, X]$.
Specializing $\widehat C_{j0}= -X^{-a_{j0}}(\sum\limits_{a\in \cA_j'} C_{j,a}X^{a})$ for every $1\le j \le r$, it follows that
$\psi(g) = g(C_0,C', X)- \sum_{1\le j \le r}X^{-a_{j0}} f_j \cdot \widetilde G_j$ with $\widetilde G_j\in \Q[C, X]_{X_1\dots X_n}$;
therefore, $g(C,X) = \sum_{1\le j \le r}X^{-a_{j0}} f_j \cdot \widetilde G_j$. Multiplying by $(X_1\dots X_n)^{m}$ for a sufficiently large $m$, we conclude that $(X_1\dots X_n)^m g(C,X)\in (f_1,\dots, f_r)$.

Then, $(f_1,\dots, f_r): (X_1\dots X_n)^{\infty}$ is a prime ideal of $\Q[C,X]$. The first statement of the lemma follows by localizing at $\Q[C]-\{0\}$.

In order to prove the second part of the lemma,
consider first a polynomial $g\in (f_1,\dots, f_r): (X_1\dots X_n)^\infty$, and let $m\in \Z_{\ge 0}$ be such that $(X_1\dots X_n)^m g \in (f_1,\dots, f_r)$. Then, $(X_1\dots X_n)^m g$ vanishes over
$V(\mathbf{f})$ and so, $g$ vanishes over $V(\mathbf{f})\cap (\overline K^*)^n$; therefore, $g\in I(V^*(\mathbf{f}))$.
Conversely, if $g\in I(V^*(\mathbf{f}))$, then it vanishes over each irreducible component of $V(\mathbf{f})$ intersecting $(\overline K^*)^n$ properly.
Then, $(X_1\dots X_n) g $ vanishes over $V(f_1,\dots, f_r)$, which implies that there exists $m\in \Z_{\ge 0}$ with $(X_1\dots X_n)^m g^m \in (f_1,\dots, f_r)$. Therefore,
$g^m\in (f_1,\dots, f_r): (X_1\dots X_n)^\infty$ and, since this is a prime ideal, it follows that $g\in(f_1,\dots, f_r): (X_1\dots X_n)^\infty$.
\end{proof}

\begin{corollary}
The affine variety $V^*(\mathbf{f})\subset \overline K^n$ is an irreducible $K$-variety of dimension $n-r$.
\end{corollary}

Taking into account that for any $K$-variety $V\subset \overline K^n$, the identity
$I(\overline{\pi(V)}) = I(V) \cap K[X_1,\dots, X_\ell]$ holds,
Lemma \ref{lem:idealV*} also enables us to characterize the vanishing ideal of the projection we want to compute:

\begin{corollary}\label{coro:idealpi} With the previous assumptions and notation,
$$I(\overline{\pi(V^*(\mathbf{f}))}) = \left((f_1,\dots, f_r): (X_1\dots X_n)^\infty\right) \cap K[X_1,\dots, X_\ell].$$
\end{corollary}

\bigskip

Let $t:= \dim(\overline{\pi(V^*(\mathbf{f}))})$.
Without loss of generality, by renaming variables, we may assume that $\{X_1,\dots, X_t\}\subset \{X_1,\dots, X_\ell\}$ is a transcendence basis of $K(\overline{\pi(V^*(\mathbf{f}))})$ over $K$ and $\{ X_{t+r+1},\dots, X_n\}$ are such that $\{X_1,\dots, X_t, X_{t+r+1}, \dots, X_n\}$ is a transcendence basis of $K(V^*(\mathbf{f}))$ over $K$. The following proposition generically allows us to deal with projections with $0$-dimensional generic fibers.

\begin{proposition} \label{prop:f_b} There is a Zariski dense open set $\mathcal{O}\subseteq \overline K^{n-t-r}$ such that, for every $b\in K^{n-t-r} \cap \mathcal{O}$, the identity
$I(\overline{\pi(V^*(\mathbf{f}))}) = ((f_1(X_1,\dots,X_{t+r}, b), \dots, f_r(X_1,\dots, X_{t+r},b)): (X_1\dots X_{t+r})^\infty)\cap K[X_1,\dots, X_\ell]$  holds.
\end{proposition}

\begin{proof} Let $\widehat X:=(X_1,\dots, X_{t+r})$. From Corollary \ref{coro:idealpi},
it is clear that $I(\overline{\pi(V^*(\mathbf{f})}))  \subset
((f_1(\widehat X, b), \dots, f_r(\widehat X,b)): (X_1\dots
X_{t+r})^\infty)\cap K[X_1,\dots, X_\ell]$ for every
$b=(b_{t+r+1},\dots, b_n)$ such that $b_i \ne 0$ for every
$t+r+1\le i \le n$.

For the converse inclusion, first note that, for every $1\le j \le r $,
$$ K[X_1,\dots, X_n]/(f_1,\dots, f_j): (X_1\dots X_n)^\infty  \simeq K[X_1,\dots, X_n]_{(X_1\dots X_n)} /(f_1,\dots, f_j)$$ $$\simeq K[Y,X_1,\dots, X_n] /(YX_1\dots X_n-1,f_1,\dots, f_j).$$
As in the first part of the proof of Lemma \ref{lem:idealV*}, we have that $(f_1,\dots, f_j): (X_1\dots X_n)^\infty $ is a prime ideal of dimension $n-j$ for every $1\le j\le r$; therefore, $YX_1\dots X_n-1$, $f_1,\dots, f_r$ is a reduced regular sequence in $K[Y,X_1,\dots, X_n]$. Moreover, the set $\{X_1,\dots, X_t, X_{t+r+1},\dots, X_n\}$ is algebraically independent modulo $(YX_1\dots X_n-1,f_1,\dots, f_r)$.
Then, by \cite[Corollary 17 and Theorem 19]{DJOSS}, there exists a $K$-definable Zariski dense open set $\mathcal{O}\subset \overline K^{n-t-r}$ containing $\{x_i\ne 0\ \forall \, 1\le i \le t+r\}$ such that for every $b\in \mathcal{O}\cap K^{n-t-r}$,
$c_b\, Y X_1\dots X_{t+r}-1,f_1(\widehat X, b) ,\dots, f_r(\widehat X, b)$
(where $c_b:=b_{t+r+1}\dots b_{n}$) is a reduced regular sequence in $K[\widehat X]$ and  $\{X_1,\dots, X_t\}$ is algebraically independent modulo each of the associated primes of the ideal this regular sequence generates.
By noticing that
$$ K[\widehat X]/(f_1(\widehat X, b) ,\dots, f_r(\widehat X, b)): (X_1\dots X_{t+r})^\infty  \simeq $$ $$\simeq K[\widehat X] /(c_b \, YX_1\dots X_{t+r}-1,f_1(\widehat X,b),\dots, f_{r}(\widehat X, b)),$$
we conclude that $(f_1(\widehat X, b) ,\dots, f_r(\widehat X, b)): (X_1\dots X_{t+r})^\infty $ is a
radical equidimensional ideal of dimension $t$ and $\{X_1,\dots, X_t\}$ is algebraically
independent modulo each of its associated primes. Then, the same holds for the ideal
$((f_1(\widehat X, b) ,\dots, f_r(\widehat X, b)): (X_1\dots X_{t+r})^\infty )\cap K[X_1,\dots,
X_\ell]$. As this ideal includes the $t$-dimensional prime ideal $I(\overline{\pi(V^*(\mathbf{f}))})$, the
equality in the statement of the proposition holds.
\end{proof}

\subsection{Free variables}

The following result gives a combinatorial condition for the algebraic independence of a subset of variables modulo the vanishing ideal of $V^*(\mathbf{f})$ that  we will use in our algorithm to compute a suitable transcendence basis of the rational fraction field of this variety.

\begin{lemma} \label{mixedvolume} Let $f_1,\dots, f_r\in K[X_1,\dots, X_n]$ be sparse polynomials supported on a family  $\cA= (\cA_1, \dots, \cA_r)$ of finite subsets of $(\Z_{\ge 0})^n$  as introduced in (\ref{eq:genericpols}), and let $\mathcal{I}= (f_1,\dots, f_r): (X_1\dots X_n)^\infty\subset K[X_1,\dots, X_n]$. Let $\{e_1,\dots, e_n\}$ be the canonical basis of $\Q^n$.

Then, the set $\{ X_{i_1},\dots, X_{i_k}\}\subset \{X_1,\dots, X_n\}$ for $1\le i_1<\dots < i_k\le n$ is algebraically independent modulo $\mathcal{I}$ if and only if $MV(\cA_1,\dots, \cA_r, \{0,e_{i_1}\},\dots, \{0,e_{i_k}\}, \Delta^{(n-r-k)})>0$.
\end{lemma}

\begin{proof} First, assume that $MV(\cA_1,\dots, \cA_r, \{0,e_{i_1}\},\dots, \{0,e_{i_k}\},\Delta^{(n-r-k)})>0$. Let $Q$ be a non-zero polynomial in the coefficients of a system supported on $(\cA,$ $ \{0,e_{i_1}\},\dots, \{0,e_{i_k}\},$ $\Delta^{(n-r-k)})$ such that for each coefficient vector $ (c,\eta)= (c_1,\dots, c_r, \eta_{r+1},\dots, \eta_n)$ with coordinates in $\Q^*$
and $Q(c,\eta) \ne 0$, the corresponding sparse system has as many isolated solutions in $(\C^*)^n$ as the mixed volume.

Let $p\in \mathcal{I} \cap K[X_{i_1},\dots, X_{i_k}]$. Without loss of generality, we may assume that $p\in \Q[C,X_{i_1},\dots, X_{i_k}]$. We have that there exist $m\in \Z_{\ge 0}$ and
a non-zero polynomial $p_0\in \Q[C]$ such that
\begin{equation}\label{eq:identity}
p_0(C) (X_1\dots X_n)^m p(C, X_{i_1},\dots, X_{i_k}) = g_1 f_1+\cdots+ g_r f_r
\end{equation}
with $g_1,\dots, g_r\in \Q[C,X_1,\dots, X_n]$.

For each $(c,\eta)$ with coordinates in $\Q^*$ such that $p_0(c) Q(c,\eta) \ne 0$, considering a solution $\xi\in (\C^*)^n$ of the system
$$f_1(c_{1}, X) = 0,\dots, f_r(c_{r}, X)= 0, \eta_{r+1,1} + \eta_{r+1,2} X_{i_1} =0,\dots,
\eta_{r+k,1} + \eta_{r+k,2} X_{i_k} =0,$$ $$ \eta_{r+k+1,0}+\sum_{1\le i \le n} \eta_{r+k+1,i}X_i=0,\dots, \eta_{n,0}+\sum_{1\le i \le n} \eta_{n,i}X_i=0$$ and specializing identity (\ref{eq:identity}) in $(c, \eta, \xi)$
we obtain that
$p(c, -\eta_{r+1,1}/\eta_{r+1,2}, \dots, -\eta_{r+k,1}/\eta_{r+k,2}) =0.$
We conclude that $p\equiv 0$.

Assume now that $\{X_{i_1},\dots, X_{i_k}\}$ is algebraically independent modulo $\mathcal{I}= I(V^*(\mathbf{f}))$. Let
$l_1,\dots, l_{n-r-k}$ be linear forms in the variables $X_1,\dots, X_n$ with coefficients in $K^*$ such that $\{X_{i_1},\dots, X_{i_k}, l_1,\dots, l_{n-r-k}\}$ is a
transcendence basis of $K(V^*(\mathbf{f}))$. Since $V^*(\mathbf{f})\cap (\overline K^*)^n$ is a dense open subset of $V^*(\mathbf{f})$, for a
generic $(\zeta_{1},\dots, \zeta_{n-r})\in (\overline K^*)^{n-r}$, we have that
$$V^*(\mathbf{f}) \cap \{ x_{i_j}= \zeta_{j}\ \forall\, 1\le j \le k\} \cap  \{l_j(x)= \zeta_{k+j} \ \forall \, 1\le j \le n-r-k\}$$ is not empty and consists of finitely many points in $(\overline K^*)^n$. These points are the common solutions in $(\overline K^*)^n$ of the system
$$f_1(X)= 0,\dots,  f_r(X)= 0,X_{i_1}-\zeta_1=0,\dots, X_{i_k}- \zeta_k=0,$$ $$ l_1(X) -\zeta_{k+1}=0,\dots, l_{n-r-k}(X)-\zeta_{n-r}=0,$$
which is supported on $(\cA_1,\dots, \cA_r, \{0, e_{i_1}\},\dots, \{ 0, e_{i_k}\}, \Delta^{(n-r-k)})$. By Bernstein's Theorem, we conclude that
 $MV(\cA_1,\dots, \cA_r, \{0, e_{i_1}\},\dots, \{ 0, e_{i_k}\}, \Delta^{(n-r-k)})>0$.
\end{proof}

\begin{remark}
For $1\le i_1<\dots< i_k\le n$, we have that $$MV_n(\cA_1,\dots, \cA_r, \{0,e_{i_1}\},\dots, \{0,e_{i_k}\},\Delta^{(n-r-k)})= MV_{n-k} (\varpi(\cA_1),\dots, \varpi(\cA_r), \varpi(\Delta)^{(n-r-k)})$$ (see \cite[Lemma 6]{ST10}), where $\varpi: \R^n\to \R^{n-k}$ is the projection to the coordinates indexed by $\{1,\dots, n\} - \{i_1,\dots, i_k\}$. This implies that, in order to determine whether a set of $k$ variables is algebraically independent or not, it suffices to compute an $(n-k)$-dimensional mixed volume.
\end{remark}

\section{Algorithms}\label{sec:algorithms}

In this section we will present an algorithm to compute the Zariski closure of the projection of $V^*(\mathbf{f})$ to the first $\ell$ coordinates, where $\mathbf{f}$ is a generic polynomial system with given supports. First, we describe some subroutines.

\subsection{Subroutines}

The first subroutine we will use, which follows from Lemma \ref{mixedvolume}, computes a transcendence basis of $K(V^*(\mathbf{f}))$ containing a transcendence basis of $K(\overline{\pi(V^*(\mathbf{f}))})$.

\bigskip

\noindent\hrulefill

\noindent \textbf{Algorithm} \texttt{TransBasis}

\medskip
\noindent  INPUT: A family $\cA= (\cA_1,\dots, \cA_r)$ of finite subsets of $(\mathbb{Z}_{\ge
0})^n$ such that $\dim (\sum_{j\in J} \cA_j )\ge \# J$ for every $J\subset \{1,\dots, r\}$.
\begin{enumerate}

\item $TB := \emptyset$

\item $k:= 1$

\item while $\#TB < n-r$ do

\begin{enumerate}

\item  If $MV(\cA_1,\dots, \cA_r, \left(\{0,e_{i_j}\}\right)_{i_j \in TB},
\{0,e_{k}\},\Delta^{(n-r-\# TB-1)})> 0$, \\
$TB := TB \cup \{k\}$.
\item $k:= k+1$
\end{enumerate}
\end{enumerate}

\noindent OUTPUT: The set $TB= \{ i_1,\dots, i_{n-r}\}$  with $i_1 < \dots < i_{n-r}$ such that $\{X_{i_1},\dots, X_{i_{n-r}}\}$
is a transcendence basis of $K (V^*(\mathbf{f}))$  over $K$  and, for each $1\le j \le n-r$,  $\{X_{i_1},\dots, X_{i_j}\}$ is a maximal
algebraically independent subset of $\{X_{ 1},\dots, X_{i_j}\}$.
\medskip

\noindent \hrulefill

\bigskip

Note that the above algorithm requires to decide whether the mixed volume of a family of finite sets is non-zero for at most $n$ families. Following \cite[Theorem 8]{DGH98}, there is a polynomial time algorithm to achieve this task based on the matroid intersection algorithm from \cite{Edmonds70}. Therefore, Algorithm \texttt{TransBasis} runs in polynomial time.

\bigskip

Without loss of generality, by renaming variables, we may assume that  the transcendence basis of $K(V^*(\mathbf{f}))$ obtained by applying algorithm \texttt{TransBasis} is $\{X_1,\dots, X_t,$ $ X_{t+r+1},\dots, X_n\}$ with $t\le \ell$ and $\ell \le t+r$. Then, $\{ X_1,\dots, X_t\}$ is a transcendence basis of $K(\overline{\pi(V^*(\mathbf{f}))})$.
Let $\mathbf{f}_b$ be the polynomial system obtained by evaluating $X_{t+r+1},\dots, X_n$ in a generic point $b$. By Proposition \ref{prop:f_b}, we may obtain $\overline{\pi(V^*(\mathbf{f}))}$ using the system $\mathbf{f}_b$. In order to do this, we will compute first a geometric resolution of $V^*(\mathbf{f}_b)$ by means of the subroutine we introduce below.

Let $k$ be a field of characteristic $0$. Algorithm \texttt{ParametricToricGeomRes}
computes a geometric resolution of the variety $V^*(\mathbf{g})$ defined from a generic sparse system $\mathbf{g}:=(g_1,\dots, g_r)$ in $k[X_1,\dots, X_{t+r}]$ with given supports $\mathcal{S}:=(\mathcal{S}_1,\dots,\mathcal{S}_r)$, provided that $\{X_1,\dots, X_t\}$ is a set of independent variables for all its components.
This subroutine is obtained by following the parametric geometric resolution algorithm from \cite[Theorem 2]{Schost03} taking $X_1,\dots, X_t$ as the parameters.

 \bigskip

\noindent \hrulefill

\noindent \textbf{Algorithm} \texttt{ParametricToricGeomRes}

\medskip

\noindent INPUT: A generic sparse system $\mathbf{g}:=(g_1,\dots, g_r)$ in $k[X_1,\dots, X_{t+r}]$ with given supports $\mathcal{S}:=(\mathcal{S}_1,\dots,\mathcal{S}_r)$ such that $\{X_1,\dots, X_t\}$ is algebraically independent modulo each associated prime of $(g_1,\dots, g_r): (X_1\dots X_{t+r})^\infty$.

\begin{enumerate}
\item Choose $\xi= (\xi_1,\dots, \xi_t) $ at random with $\xi_i\in \Z-\{0\}$ for every $1\le i \le t$.
\item Compute a geometric resolution of the common solutions of $\mathbf{g}(\xi, X_{t+1},\dots, X_{t+r})$ in $(\overline k^*)^r$.
\item Obtain an slp encoding the polynomials in $\mathbf{g}$.
\item Apply a symbolic Newton-Hensel lifting (in the parameters $X_1,\dots, X_t$) to the geometric resolution obtained in step 2  with precision $2 MV(\mathcal{S}, \Delta^{(t)})$.
\item  Applying Pad\'e approximation to the output of the previous step, recover numerators and denominators in $k[X_1,\dots, X_t]$ for the coefficients of the polynomials in the geometric resolution of $V^*(\mathbf{g})$.
\end{enumerate}

\noindent OUTPUT: A geometric resolution of $V^*(\mathbf{g})$ with free variables $X_1,\dots, X_t$.

\medskip

\noindent\hrulefill

\bigskip

Before estimating the complexity of the previous algorithm, we present a simple example to illustrate how the algorithm works.

\bigskip

\noindent \textbf{Example.}
Let $\mathbf{g}$ be the following sparse system supported on $\mathcal{S}=(\mathcal{S}_1, \mathcal{S}_2)$, where $\mathcal{S}_1=\{(0,0,0), (1,1,0),(0,1,1)\}$ and $\mathcal{S}_2=\{(0,0,0), (2,1,1),(0,2,0),(1,1,1)\}$:

$$\mathbf{g}:=\begin{cases}
g_1= 2+3X_1X_2-X_2X_3\\
g_2= -1+2X_1^2X_2X_3+2X_2^2+X_1X_2X_3
\end{cases}$$
Here, $\{X_1\}$ is algebraically independent modulo $(g_1,g_2): (X_1X_2X_3)^\infty$.

In step 1, the algorithm chooses a value $\xi_1\in \Z-\{0\}$ at random and specializes $X_1= \xi_1$. Setting $\xi_1=1$, we obtain the system
$$\mathbf{g}_1:=\begin{cases}
g_{11}= 2+3X_2-X_2X_3\\
g_{21}= -1+3X_2X_3+2X_2^2
\end{cases}$$

Now, the algorithm computes a geometric resolution of the common
zeros of $\mathbf{g}_1$ in $(\C^*)^2$. This is a generic system
supported on $(\{(0,0),(1,0), (1,1)\}, \{(0,0),(1,1),(2,0)\})$.
Then, in order to solve it, we can apply the algorithm from
\cite[Section 5]{JMSW09}. Choosing $X_3$ as separating linear
form, we obtain
\begin{itemize}
\item $q_{X_3} (Y)= Y^2-\frac{12}{5}Y-\frac{1}{5}$
\item $w_2(Y)= -\frac54 Y - \frac34$
\item $w_3(Y) = Y$
\end{itemize}

From this geometric resolution, the algorithm computes a sufficiently good approximation of a geometric resolution of $V^*(\mathbf{g})$.

A geometric resolution of $V^*(\mathbf{g})$ consists in univariate polynomials with coefficients that are rational functions in the parameter $X_1$ with numerators and denominators of degrees bounded by $\deg V^*(\mathbf{g})$ (see \cite[Theorem 1]{Schost03}). These rational functions can be regarded as power series in the variable $X_1$ centered at $\xi_1=1$, and therefore, they can be recovered by means of Pad\'e approximation from sufficiently many terms of their expansions. The precision required to do this equals $2 \deg V^*(\mathbf{g}) = 2 MV(\mathcal{S}, \Delta)$.
 Since $MV(\mathcal{S}, \Delta)=6$, in step $4$, the algorithm applies a Newton-Hensel lifting (in the parameter $X_1$), as explained in \cite[Section 4.2]{Schost03}, to the geometric resolution $(q_{X_3}, w_2, w_3)$ with precision $12$, obtaining
\begin{itemize}
\item $\widehat q_{X_3} (Y)= Y^2 + q_1 Y+  q_0 $
\item $\widehat w_2(Y)=  w_{21} Y + w_{20} $
\item $\widehat w_3(Y) = w_{31} Y +w_{30}$
\end{itemize}
where
\begin{itemize}
\item $q_1= -\frac{12}{5}-\frac{18}{5}(X_1-1)+\frac{18}{25}(X_1-1)^2-\frac{24}{25}(X_1-1)^3+\frac{168}{125}(X_1-1)^4
    -\frac{48}{25}(X_1-1)^5+\frac{1728}{625}(X_1-1)^6-\frac{2496}{625}(X_1-1)^7+\frac{18048}{3125}(X_1-1)^8
    -\frac{26112}{3125}(X_1-1)^9+\frac{188928}{15625}(X_1-1)^{10}-\frac{273408}{15625}(X_1-1)^{11}+\frac{1978368}{78125}(X_1-1)^{12}$
\item $q_0= -\frac15-\frac{16}{5}(X_1-1)+\frac{119}{25}(X_1-1)^2-\frac{174}{25}(X_1-1)^3+\frac{1264}{125}(X_1-1)^4-
    \frac{1832}{125}(X_1-1)^5+\frac{13264}{625}(X_1-1)^6-\frac{768}{25}(X_1-1)^7+\frac{138944}{3125}(X_1-1)^8
    -\frac{201088}{3125}(X_1-1)^9+\frac{1455104}{15625}(X_1-1)^{10}-\frac{2105856}{15625}(X_1-1)^{11}+
    \frac{15238144}{78125}(X_1-1)^{12}$
\item $w_{21}= -\frac{5}{4}-\frac{5}{2}(X_1-1)-(X_1-1)^2$
\item $w_{20}= -\frac34-\frac34(X_1-1) $
\item $w_{31} = 1$
\item $w_{30} =0$
\end{itemize}
Finally, Pad\'e approximation is applied, following \cite[Corollaries 5.21 and 6.48]{vzGG99}, to each of the coefficients previously computed in order to obtain the numerators and denominators of the coefficients in the geometric resolution of $V^*(\mathbf{g})$ associated to the linear form $X_3$:
\begin{itemize}
\item $Q_{X_3}(Y) =  Y^2 + \frac{-12X_1^3-6X_1^2+6X_1}{4X_1^2+2X_1-1} Y+ \frac{-9X_1^2+8}{4X_1^2+2X_1-1}$
\item $W_2(Y)=  (-X_1^2-\frac12X_1+\frac14) Y -\frac{3}{4}X_1 $
\item $W_3(Y) =  Y$
\end{itemize}

\bigskip

Now we explain how the different steps of the algorithm are carried out in general and estimate the number of operations in $k$ it performs. We will use the notation $N=\sum_{1\le j \le r}\# \mathcal{S}_j$ and $d=\max_{1\le j \le r} \deg(g_j)$.

In Step 2, the algorithm computes the sparse encoding of $\mathbf{g}(\xi, X_{t+1},\dots, X_{t+r})$
within $O(r(t+r)N \log(d))$ operations. Then, this system is solved using the procedure from \cite[Section 5]{JMSW09} within complexity $O(r^3 N \log (d) M(\mathcal{M}) M(MV(\varpi(\mathcal{S})))(M(MV(\varpi(\mathcal{S})))+  M(\omega_{\max} \sum_{1\le h\le r} MV(\varpi(\mathcal{S}_j)_{j\ne h}, \Delta))))$, where
$\varpi$ is the projection to the last $r$ coordinates, $\omega_{\max}$ is the maximum of the values taken by a generic lifting function $\omega$ for $\varpi(\mathcal{S}):=(\varpi(\mathcal{S}_1), \dots, \varpi(\mathcal{S}_r ))$, and
$\mathcal{M}:= \max \{ || \mu || \}$, the maximum ranging over all primitive normal vectors to the mixed cells in the fine mixed subdivision of $\varpi(\mathcal{S})$ induced by $\omega$.

In Step 3, an slp of length $O((t+r)N\log(d))$ encoding the  polynomials $\mathbf{g}$ is obtained from their sparse representation.

The next step is performed modifying the procedure in \cite[Section 4.2]{Schost03} in order to use straight-line programs for computations with truncated multivariate power series. We represent each of these series as the vector of its homogeneous components and these components by means of straight-line programs. The required precision is $2 \deg(V^*(\mathbf{g}))= 2 MV(\mathcal{S}, \Delta^{(t)})$ (see \cite[Theorem 1]{Schost03} and \cite[Lemma 1]{HJS}). Therefore, this step is done within complexity
$O((r(t+r)N\log(d)+ r^4) M (MV(\varpi(\mathcal{S}))) MV(\mathcal{S}, \Delta^{(t)})^2)$ and produces an slp of the same order encoding the homogeneous components of the coefficients of the output.

Finally, the Pad\'e approximation to obtain the coefficients of the geometric resolution is done by reducing it to a univariate problem and solving it by means of subresultant computations following \cite[Corollaries 5.21 and 6.48]{vzGG99}. This step adds $O(rMV(\mathcal{S}, \Delta^{(t)})^{\Omega+2})$ operations to the previous complexity and $O(rMV(\mathcal{S}, \Delta^{(t)})^{\Omega+1})$ to the slp length.

Therefore, we have the following complexity result:

\begin{lemma}\label{comp:PTGR}
Let $\mathbf{g}:=(g_1,\dots, g_r)$ in $k[X_1,\dots, X_{t+r}]$ be a generic sparse system with  supports $\mathcal{S}:=(\mathcal{S}_1,\dots,\mathcal{S}_r)$ such that $\{X_1,\dots, X_t\}$ is algebraically independent modulo each associated prime of $(g_1,\dots, g_r): (X_1\dots X_{t+r})^\infty$. Algorithm \verb"ParametricToricGeomRes" computes a geometric resolution of $V^*(\mathbf{g})$.
With the previous notation, the total number of operations in $k$ performed by the algorithm is of order
\begin{equation*}
\begin{array}{c}
O((r^3+rt) N \log(d)  M(\mathcal{M}) M(MV(\varpi(\mathcal{S}))) (MV(\mathcal{S}, \Delta^{(t)})^2 +{} \\[2mm] {} + M(\omega_{\max} \sum_{1\le h\le r} MV(\varpi((\mathcal{S}_j)_{j\ne h}), \Delta)))+r MV(\mathcal{S}, \Delta^{(t)})^{\Omega + 2})
\end{array}\end{equation*}
The algorithm produces an slp of length
\begin{equation*}
O(rMV(\mathcal{S}, \Delta^{(t)})^2((t+r)N\log(d)+ r^3) M
(MV(\varpi(\mathcal{S})))+ MV(\mathcal{S},
\Delta^{(t)})^{\Omega-1})
\end{equation*} for the coefficients of the output.
\end{lemma}

The last step of our main algorithm consists in describing the projection  to a coordinate subspace of an equidimensional variety of dimension $t$ given by a geometric resolution in the case that the projection has the same dimension $t$. To do this, we apply the subroutine described below.

\bigskip

Let $V \subset \overline{k}^{t+r}$ be an equidimensional variety of dimension $t$ definable over $k$ and such that for
each irreducible component $W$
 of $V$, $I(W) \cap k[X_1,\dots, X_t] = \{0\}$ holds. Consider the projection $\pi: \overline{k}^{t+r} \to \overline{k}^{\ell}$ where $\ell > t$, $\pi (x_1,\dots, x_{t+r}) = (x_1,\dots, x_\ell)$. Note that $\{X_1,\dots, X_t\}$ are free variables with respect to each irreducible component of  $\overline{\pi(V)}$.
Let $\K:= k(X_1,\dots, X_t)$.
Suppose $\lambda \in k[X_{t+1},\dots, X_{t+r}]$ is a primitive element for $\K \otimes k[V]$ and let
$(q_\lambda, w_{t+1}, \dots, w_{t+r})\in \K[Y]^{r+1}$ be the associated geometric resolution. Let $D$ be the dimension of $\K \otimes  k[V]$ as $\K$-vector space.

Let $\mu = \mu_{t+1} X_{t+1}+ \dots + \mu_{\ell} X_{\ell}$ be a primitive element for $\K  \otimes  k[\overline{\pi(V)}]$.
As $I(\overline{\pi(V)}) = I(V) \cap k[X_1,\dots, X_\ell]$, to find the minimal polynomial of $\mu$ with respect to $\overline{\pi(V)}$, it suffices to find
a polynomial $q_{\mu} \in \K[Y]$ of minimal degree such that $q_\mu (\mu)\in \K \otimes I(V)$.
Then, $\delta := \deg_Y(q_\mu)$ is the dimension of $\K\otimes k[\overline{\pi(V)}]$ as a $\K$-vector space and so, for each $t+1 \le j \le \ell$,
in order to obtain a polynomial $v_j$ such that $X_{j} = v_j(\mu)$ in $ \K  \otimes k[\overline{\pi(V)}]$ it suffices to find a linear combination
$X_{j} = \sum_{i=0}^{\delta - 1} v_{ji} \mu^i$ of $\{1,\mu, \dots, \mu^{\delta - 1}\}$ in  $\K  \otimes  k[V]$. To do this, we use the basis
$B_\lambda:=\{1, \lambda, \dots, \lambda^{D-1}\}$ of $\K \otimes  k[V]$.

In order to compute the geometric resolution of $\overline{\pi(V)}$ associated with $\mu$, we first look for the minimal power $\mu^{\delta}$ which is a $\K$-linear combination of $\{1,\mu,\dots, \mu^{\delta-1}\}$ in $\mathbb{K}\otimes k[V]$.
Since $X_{j} = w_{j}(\lambda)$ for every $t+1\le j\le t+r$, we have that
$ \mu = \sum_{j=t+1}^{\ell} \mu_{j} w_{j}(\lambda)= p_\mu(\lambda)$, where $p_\mu(Y):=\sum_{j=t+1}^{\ell} \mu_{j} w_{j}(Y)$ and, for every $i\in \N$, $\mu^i = p_\mu(\lambda)^i = (p_\mu(Y)^i \mod q_\lambda(Y))|_{Y=\lambda}$. Therefore, $\delta$ equals the rank of the $D\times D$ matrix whose columns are the coefficients of the polynomials $(p_\mu(Y)^i \mod q_\lambda(Y))$ for $i=0,\dots, D-1$.

Then, we obtain the minimal polynomial $q_\mu(Y):=Y^\delta + \sum_{i=0}^{\delta -1} q_{\mu,i} Y^i$ of $\mu$ and the polynomials $v_j(Y)=\sum_{i=0}^ {\delta-1} v_{ji} Y^i$, $t+1\le j \le \ell$, which form the geometric resolution of $\overline{\pi(V)}$,  by solving the linear systems obtained by equating the coefficients of the different powers of $\lambda$ in the identities
$$p_{\mu^\delta} (\lambda) = \sum_{i=0}^{\delta-1} (-q_{\mu,i}) p_{\mu^i}(\lambda) \quad \hbox{ and }\quad w_{j}(\lambda) = \sum_{i=0}^{\delta-1} v_{ji}\,  p_{\mu^i}(\lambda),\ t+1\le j \le \ell.$$

Summarizing, with the previous notation and hypothesis, we have:

\bigskip

\noindent\hrulefill

\noindent \textbf{Algorithm} \texttt{GeomResProj}

\medskip

\noindent  INPUT: A geometric resolution  $(q_\lambda, w_{t+1}, \dots, w_{t+r})$ of $V$ with free variables $X_1,\dots, X_t$, and a linear form $\mu = \sum_{j=1}^{\ell-t} \mu_{t+j} X_{t+j} \in k[X_{t+1},\dots, X_{\ell}]$ which is
a primitive element for $\K  \otimes  k[\overline{\pi(V)}]$.
\begin{enumerate}
\item\label{1} Set $p_{\mu^0}(Y):=1$ and $p_{\mu}(Y):= \sum_{h=0}^{D-1} ( \sum_{j=t+1}^{\ell} \mu_{j} w_{j, h}) Y^h$, where $(w_{j,0},\dots, w_{j, D-1})=: \mathbf{w}_{j}$ is the vector of coefficients of $w_{j}(Y)$ for $j=t+1,\dots, t+r$.
\item For $i=2,\dots, D$, compute $p_{\mu^i}(Y):= ( p_\mu(Y) \cdot p_{\mu^{i-1}}(Y) \mod q_\lambda(Y))$
\item Compute $\delta:= \text{rank}(\mathbf{p}_{\mu^0},\mathbf{p}_\mu,\dots, \mathbf{p}_{\mu^{D-1}})$, where $\mathbf{p}_{\mu^i}\in \mathbb{K}^{D\times 1}$ denotes the vector of coefficients of $p_{\mu^i}$.

\item Set $\mathbf{P}:= (\mathbf{p}_{\mu^0},\mathbf{p}_\mu,\dots, \mathbf{p}_{\mu^{\delta-1}})\in \mathbb{K}^{D\times \delta}$.

\item Solve the linear systems $\mathbf{P} \cdot \mathbf{q} =  \mathbf{p}_{\mu^\delta}$ and
$\mathbf{P} \cdot \mathbf{v}_j =  \mathbf{w}_{j}$, $t+1\le j \le \ell$, to obtain
 $\mathbf{q}:= (q_0,\dots, q_{\delta-1})$ and $\mathbf{v}_j:= (v_{j,0},\dots, v_{j,\delta-1})$

\item Set $q_\mu(Y):= Y^\delta - \sum_{i=0}^{\delta-1} q_i Y^i$ and $v_j(Y):=\sum_{i=0}^{\delta-1} v_{ji} Y^i$  for every $t+1\le j \le \ell$.

\end{enumerate}
\noindent OUTPUT: The geometric resolution  $(q_\mu, v_{t+1}, \dots, v_{\ell})$ of the projection
$\overline{\pi(V)} \subset \overline{k}^\ell$ with free variables $X_1,\dots, X_t$ associated to the linear form $\mu$.

\noindent\hrulefill

\bigskip

The correctness of the procedure follows from our previous arguments. Now we estimate its complexity.
Assume that the input polynomials are encoded in dense form as degree $D$ univariate polynomials in $k(X_1,\dots, X_t)[Y]$ and their coefficients are encoded by an slp over $k$ of length $L$.

In Step 1, the algorithm computes an slp encoding the coefficients of $p_\mu$ of length bounded by $L+ 2 D (\ell-t)$.

In order to fulfill Step 2, we first compute recursively the powers $Y^h$ for $h=D,\dots, 2D-2$ modulo $q_\lambda(Y)$ from the coefficients of $q_\lambda$ and then, we obtain an slp of length $O(L +D(\ell-t)+D^3)$ for the coefficients of $p_{\mu^i}$ for $i=2,\dots, D$ by expanding the product $p_\mu(Y)\cdot p_{\mu^{i-1}}(Y)$ and substituting the powers of $Y$ by their previously computed expressions.

Step 3 is done probabilistically by choosing a point $x=(x_1,\dots, x_t)$ at random, evaluating the involved rational functions at this point within $O(L +D(\ell-t)+D^3)$ operations in $k$ and finally
 computing the rank $\delta$ of $(\mathbf{p}_{\mu^0}(x),\mathbf{p}_\mu(x),\dots, \mathbf{p}_{\mu^{D-1}}(x))$ within $O(D^\omega)$ operations in $k$, with $\omega<3 $ (see \cite[Chapter 2, Sec. 2]{BP94}).

To solve the linear systems involved in Step 5, the algorithm computes the invertible matrix $\mathbf{P}^t\mathbf{P}$ within $O(D\delta^2)$ operations in $k$, its adjoint matrix and determinant within $O(\delta^\Omega)$ additional operations, and the products $\textrm{adj}(\mathbf{P}^t\mathbf{P}) \textbf{p}_{\mu^\delta}$ and $\textrm{adj}(\mathbf{P}^t\mathbf{P}) \textbf{w}_{j}$ for $t+1\le j\le \ell$ with $O(\delta^2 (\ell-t))$.

Adding the previous estimates, we conclude that the algorithm produces an slp of length $O(L+D^3+ D\delta(\ell-t)+ \delta^\Omega)$ over $k$ within a complexity of the same order.  Taking into account that $\delta \le D$, we have:

\begin{lemma}\label{comp:GRP}
Let $V \subset \overline{k}^{t+r}$ be an equidimensional variety of dimension $t$ definable over $k$ and such that for
each irreducible component $W$
 of $V$, $I(W) \cap k[X_1,\dots, X_t] = \{0\}$ holds.  With the previous notation,
 given a geometric resolution of $V$  with free variables $X_1,\dots, X_t$, and a linear form $\mu $ which is
a primitive element for $\K  \otimes  k[\overline{\pi(V)}]$,  Algorithm  \verb"GeomResProj" computes
the geometric resolution of the projection
$\overline{\pi(V)} \subset \overline{k}^\ell$ with free variables $X_1,\dots, X_t$ associated to the linear form $\mu$ within
$O(L+D^\Omega + D^2(\ell - t))$ operations in $k$. The output of the algorithm is encoded by an slp of length of the same order.
\end{lemma}

\subsection{An algorithm to find the projection}

Here we present a probabilistic algorithm
that, from a fixed family $\cA=(\cA_1,\dots, \cA_r)$ of finite sets $(\mathbb{Z}_{\ge 0})^n$ and a fixed family of
variables $X_1, \dots,X_\ell$, obtains a geometric resolution of
$\overline{\pi(V^*(\mathbf{f}))} \subset \overline K^\ell$, where $\mathbf{f}=(f_1,\dots, f_r)$ is defined in equation (\ref{eq:genericpols}) and $\pi:\overline K^n \to \overline K^\ell$ is the projection $\pi(x_1,\dots, x_n) = (x_1,\dots, x_\ell)$.

\bigskip

\noindent\hrulefill

\noindent \textbf{Algorithm} \texttt{$K$-Projection}

\medskip
\noindent  INPUT: A family $\cA= (\cA_1,\dots, \cA_r)$ of finite
subsets of $(\mathbb{Z}_{\ge 0})^n$  such that $\dim (\sum_{j\in J} \cA_j )\ge \# J$ for every $J\subset \{1,\dots, r\}$ and a set of variables $\{
X_1, \dots, X_\ell \} \subset \{ X_1, \dots, X_n \}$.
\begin{enumerate}
\item Apply Algorithm \texttt{TransBasis} to the family $\cA$.
Without loss of generality, we suppose the transcendence basis
obtained is $\{X_{1},\dots, X_{t}, X_{t+r+ 1}, \dots, X_{n}\}$ with $t \le \ell$ and $t+r+1 >
\ell$.
\item Choose randomly  in $\Z$ the entries of a vector $b = (b_{t+r+1},\dots, b_{n})$
and of a vector $(\lambda_{t+ 1},\dots, \lambda_{t+r})$.
\item Obtain the sparse representation of the system of polynomials

$\mathbf{f}_b= (f_1(X_1,\dots,X_{t+r}, b), \dots, f_r(X_1,\dots, X_{t+r},b))$ in $K[X_1, \dots,
X_{t+r}]$.
\item Apply Algorithm \texttt{ParametricToricGeomRes} to the system $\mathbf{f}_b$ and the variables $X_1,\dots, X_t$ to obtain  the
geometric resolution $(q_\lambda, w_{t+1}, \dots, w_{t+r})$ of the
variety $V^*(\mathbf{f}_b)$  with free variables $X_1,\dots, X_t$ associated to the linear form
$\lambda = \lambda_{t+ 1} X_{t+ 1}+ \dots + \lambda_{t+r} X_{t+r}$.
\item Choose randomly in $\Z$ the entries of a vector  $ (\mu_{t+ 1},\dots, \mu_{\ell}) .$
\item Apply Algorithm \texttt{GeomResProj} to
  $(q_\lambda, w_{t+1}, \dots, w_{t+r})$  and $\mu = \mu_{t+ 1} X_{t+1}+\dots+
\mu_{\ell} X_\ell$.
\end{enumerate}
\noindent OUTPUT: A geometric resolution $(q_\mu, v_{t+1}, \dots, v_{\ell})$
of  $\overline{\pi(V^*(\mathbf{f}))}
\subset \overline{K}^\ell$, where $\pi: \overline K^n\to \overline K^\ell$ is the projection  to the first coordinates.

\noindent\hrulefill

\bigskip

\begin{theorem}\label{thm:K-Proj}
Given  a family $\cA= (\cA_1,\dots, \cA_r)$ of finite
subsets of $(\mathbb{Z}_{\ge 0})^n$  such that $\dim (\sum_{j\in J} \cA_j )\ge \# J$ for every $J\subset \{1,\dots, r\}$ and the projection $\pi: \overline K^n\to \overline K^\ell$ to the first coordinates, Algorithm \emph{\texttt{$K$-Projection}} is a probabilistic procedure that computes a geometric resolution of $\overline{\pi(V^*(\mathbf{f}))}$ for the sparse system $\mathbf{f}$ supported on $\cA$ with indeterminate coefficients within $$O( (n^2 + r^3) N \log(d) M(\mathcal{D}) \Xi (\mathcal{D}^2 + M(\mathcal{E}) )  +r\mathcal{D}^{\Omega+2})$$
 operations in $K$, where $N= \sum_{1\le j \le r} \# \cA_j$, $d:=\max_{1\le j \le r} \{ \deg_X(f_j)\}$, $\mathcal{D}=MV(\cA, \Delta^{(n-r)})$, $\mathcal{E}= \sum_{1\le h \le r} MV( ((\cA_j)_{j\ne h}), \Delta^{(n-r+1)})$,  and $\Xi$ is a constant measuring the size of certain combinatorial objects  involved at intermediate computations and associated to the family of supports $\cA$.
 \end{theorem}

\begin{proof}
Since in our complexity estimates we only take into account the number of operations in $K$ (and not mixed volume or mixed subdivision computations), to obtain the overall complexity of the algorithm it suffices to add the complexities of steps 3, 4 and 6.

Step 3 can be done within $O(n^2 N\log(d))$ operations in $K$.

The complexity of Step 4 is the already stated for Algorithm
\texttt{ParametricToricGeomRes} in Lemma \ref{comp:PTGR}. Note
that for generic $b$, the system $\mathbf{f}_b$ is a generic
polynomial system supported on $\mathcal{S}:=(\mathcal{S}_1,\dots,
\mathcal{S}_r)$, where $\mathcal{S}_j\subset (\Z_{\ge 0})^{t+r}$
is the projection of $\cA_j$ to the first $t+r$ coordinates for
every $1\le j \le r$. Moreover, by Bernstein's theorem,
$MV(\mathcal{S}, \Delta^{(t)}) \le MV(\cA, \Delta^{(n-r)})$,
$MV(\varpi(\mathcal{S}))\le MV(\cA, \Delta^{(n-r)})$ and, for
every $1\le h \le r$, $MV(\varpi((\mathcal{S}_j)_{j\ne h}),
\Delta)\le MV(((\cA_j)_{j\ne h}), \Delta^{(n-r+1)})$. We take
$\Xi$ such that $M(\mathcal{M}) M(\omega_{\max}) \le \Xi$.

Finally, the complexity of Step 6 follows from Lemma \ref{comp:GRP}. Note that, here, $D\le MV(\cA, \Delta^{(n-r)})$ and $L$ is the length of the slp computed in Step 4 according to Lemma \ref{comp:PTGR}.
\end{proof}

\subsection{Example}\label{sec:example}

Consider a sparse system of $2$ polynomials in $5$ variables supported on
$\cA=(\cA_1, \cA_2)$, where $\cA_1=\{(0,0,0,0,0), (1,1,1,0,0),(2,0,0,4,2),(0,0,0,8,4)\}$ and \newline $\cA_2=\{(1,0,1,1,2),(0,1,2,5,4),(1,3,0,5,4)\}$, with indeterminate coefficients:
$$\mathbf{f}:=\begin{cases}
f_1= C_{11}+C_{12}X_1X_2X_3+C_{13} X_1^2 X_4^4 X_5^2+ C_{14} X_4^8 X_5^4\\
f_2= C_{21} X_1 X_3 X_4 X_5^2+ C_{22} X_2 X_3^2 X_4^5 X_5^4 + C_{23} X_1 X_2^3 X_4 ^5 X_5^4
\end{cases}$$
and the projection $\pi: \A^5 \to \A^3$, $\pi(x_1,x_2,x_3,x_4,x_5)=(x_1,x_2,x_3)$. We are going to show a geometric resolution of $\overline{\pi(V^*(\mathbf{f}))}$ following Algorithm \texttt{$K$-Projection}.

\bigskip

First, we apply Algorithm \texttt{TransBasis}
and we obtain that $\{X_1,X_2,X_4\}$ is a transcendence basis of $K(V^*(\mathbf{f}))$, and so, $\{X_1,X_2\}$ is a transcendence basis of $K(\overline{\pi(V^*(\mathbf{f}))})$.

Now, the algorithm chooses a value $b$ at random and specializes $X_4=b$. Set $b=1$. The specialized system is
$$\mathbf{f}_1:=\begin{cases}
f_{11}= C_{11}+C_{12}X_1X_2X_3+C_{13} X_1^2 X_5^2+ C_{14} X_5^4\\
f_{21}= C_{21} X_1 X_3  X_5^2+ C_{22} X_2 X_3^2 X_5^4 + C_{23} X_1 X_2^3  X_5^4
\end{cases}$$

The next step is to apply Algorithm \texttt{ParametricToricGeomRes} with free variables $X_1,X_2$. We choose $\lambda= X_5$ as the primitive element to obtain the geometric resolution:

\begin{itemize}
\item $q_{X_5}(Y)= Y^{10}+\frac{2C_{13}X_1^2}{C_{14}}Y^8+\frac{C_{13}^2X_1^4+2C_{11}C_{14}}{C_{14}^2}Y^6
+\frac{(-C_{12}C_{21}C_{14}+2C_{11}C_{22}C_{13})X_1^2}{C_{22}C_{14}^2}Y^4+ {}$\\ ${}\qquad \quad +\frac{-C_{12}C_{21}C_{13}X_1^4+C_{11}^2C_{22}+C_{12}^2C_{23}X_1^3X_2^4}{C_{22}C_{14}^2}Y^2-\frac{C_{12}C_{21}C_{11}X_1^2}{C_{22}C_{14}^2}$
\item $w_3(Y)= -\frac{C_{14}}{C_{12}X_1X_2} Y^4 - \frac{C_{13}X_1}{C_{12}X_2} Y^2 - \frac{C_{11}}{C_{12} X_1 X_2}$
\item $w_5(Y)=Y$
\end{itemize}

Finally, Algorithm \texttt{GeomResProj} is applied to $q_{X_5}, w_3, w_5$ and a primitive element $\mu$ for $K(\overline{\pi(V^*(\mathbf{f}))})$. In this case, we take $\mu= X_3$ and obtain

\begin{itemize}
\item $q_{X_3}(Y)= Y^5+\frac{C_{11}}{C_{12}X_1X_2}Y^4+\frac{2C_{12}C_{23}X_1X_2^4-C_{13}C_{21}X_1^2}{C_{12}C_{22}X_2^2}Y^3+
\frac{2C_{23}C_{22}C_{11}X_2^4+C_{21}^2C_{14}X_1}{C_{12}C_{22}^2X_2^3}Y^2+$\\
 ${}\qquad\quad + \frac{C_{12}C_{23}^2X_1^2X_2^4-C_{13}C_{21}C_{23}X_1^3}{C_{12}C_{22}^2}Y+\frac{C_{23}^2C_{11}X_1X_2^3}{C_{22}^2C_{12}}
$
\item $v_3(Y)=Y$
\end{itemize}

\subsection{Computation for generic rational coefficients.}

We are now going to show that, for a system with generic
rational coefficients, the same steps as the ones in Algorithm  \texttt{$K$-Projection} can be performed
using these rational coefficients to obtain a geometric resolution of the Zariski closure of the projection of the associated variety.

\bigskip

Let $\{X_1,\dots, X_t, X_{t+r+1},\dots, X_n\}$ be a transcendence basis of $K(V^*(\mathbf{f}))$ such that $\{X_1,\dots, X_t\}$ is a maximal algebraically independent subset of $\{X_1,\dots, X_\ell\}$.

 As shown in \cite[Appendix A]{DJOSS}, there exists a non-empty Zariski open set $\cU_1\subset \A^N$ such that if $c=(c_1,\dots, c_r)\in \cU_1\cap \Q^N$, then:
\begin{itemize}
\item the ideal $I_c:=(f_1(c,X), \dots, f_r(c,X)): (X_1\dots X_n)^\infty$ is radical equidimensional of dimension $n-r$,
\item $\{X_1,\dots, X_t, X_{t+r+1},\dots, X_n\}$ is algebraically independent modulo each of the associated primes of $I_c$,
\item for every $1\le k\le \ell-t$, $\{X_1,\dots, X_t, X_{t+k}\}$
    is algebraically dependent modulo $I_c$.
\end{itemize}

Assume $c\in \cU_1\cap \Q^N$ and consider the variety $V^*(\mathbf{f}(c)) \subset \A^n$. Let $W$ be an irreducible component of $V^*(\mathbf{f}(c))$. We have that $\dim(W) = n-r$ and $\{X_1,\dots, X_t,$ $ X_{t+r+1},\dots, X_n\}$ is a transcendence basis of $\Q(W)$. Then, the projection of $W$ over the last $n-t-r$ coordinates is a dominant map. Therefore, there is a Zariski open subset $\cO_W\subset \A^{n-t-r}$ such that, for every $b\in \cO_W \cap \Q^{n-t-r}$:
\begin{itemize}
\item $W_b:= W \cap \{x_{t+r+1}=b_{t+r+1},\dots, x_n=b_{n}\}$ is an equidimensional variety of dimension $t$,
\item $\{X_1,\dots, X_t\}$ is algebraically independent modulo $I(W_b)$.
\end{itemize}
Then, for every $b\in \cO_W\cap \Q^{n-t-r}$, the identity $\overline{\pi(W)} = \overline{\pi(W_b)}$ holds.

\bigskip

As the dimension of the set $ \partial (\mathbf{f}(c)) :=
V^*(\mathbf{f}(c)) - \{ x \in (\C^*)^n : \mathbf{f}(c,x)=0\}$ is
less than $n-r$, for every $\{i_1,\dots, i_t\} \subset \{1,\dots,
t+r\}$, there exists a non-zero polynomial $p_{i_1,\dots,
i_t}(X_{i_1},\dots, X_{i_t}, X_{t+r+1},\dots, X_n)$ vanishing
identically on this set. Then, there is a non-empty Zariski open
set $\cO_1\subset \A^{n-t-r}$ such that for every $b\in \cO_1$,
the dimension of $\partial (\mathbf{f}(c)) \cap \{x_{t+r+1}=
b_{t+r+1},\dots, x_n=b_{n}\}$ is less than $t$.

Then, for every $b\in \cO_1 \cap \bigcap_W \cO_W\cap (\Q^*)^{n-t-r}$, we have that
$$V^*(\mathbf{f}(c)) \cap \{x_{t+r+1}=  b_{t+r+1},\dots, x_n=b_n\} =
\overline{ \{ \hat x \in (\C^*)^{t+r} : \mathbf{f}(c,\hat x, b) =0\}\times \{ b\}}$$
and
$\pi(V^*(\mathbf{f}(c))) = \bigcup_W \pi(W) = \bigcup_W \pi(W_b) = \pi (V^*(\mathbf{f}(c)) \cap \{x_{t+r+1}= b_{t+r+1},\dots, x_n=b_n\});$
therefore,
$$\overline{\pi(V^*(\mathbf{f}(c)))} =\overline{ \pi (\{ \hat x \in (\C^*)^{t+r} : \mathbf{f}(c,\hat x, b) =0\}\times \{ b\})}.$$

\bigskip

For $j=1,\dots, r$, let $\mathcal{S}_j\subset (\Z_{\ge 0})^{t+r}$  be the projection of $\cA_j$ to the first $t+r$ coordinates. Then $f_j= \sum_{\hat a \in \mathcal{S}_j}  (\sum_{(\hat a,\tilde a)\in \cA_j} C_{j,(\hat a,\tilde a)} \tilde X^{\tilde a}) \hat X^{\hat a} $.
Algorithm \texttt{ParametricToricGeomRes} works for generic sparse polynomial systems $\mathbf{g}$ supported on $\mathcal{S} = (\mathcal{S}_1,\dots, \mathcal{S}_r)$,  that is, there is a polynomial $p_{\mathcal{S}}$ in the coefficients of the system such that it computes a geometric resolution of $V^*(\mathbf{g}(\hat c))$ for every vector of coefficients $\hat c$ with $p_{\mathcal{S}}(\hat c)\ne 0$. Let $\cU_2\subset \A^N$ be a non-empty Zariski open set such that, for every $c:=(c_1,\dots, c_r)$, the polynomial $ p_{\mathcal{S}}((\sum_{(\hat a,\tilde a)\in \cA_j} c_{j,(\hat a,\tilde a)} \tilde X^{\tilde a})_{1\le j \le r, \hat a\in \mathcal{S}_j})$  does not vanish identically.

For $c\in \cU_2\cap \Q^N$, there exists a non-empty Zariski open set $\cO_2\subset \A^{n-t-r}$ such that for every $b\in \cO_2 \cap \Q^{n-t-r}$, the algorithm \texttt{ParametricToricGeomRes} can be applied to the system $\mathbf{f}(c,\hat x, b)$.

We conclude that, for coefficient vectors $c\in \cU_1 \cap \cU_2 \cap (\Q^*)^N$, a  probabilistic algorithm \texttt{$\Q$-Projection} which follows the same steps as  Algorithm \texttt{$K$-Projection} can be applied in order to compute a geometric resolution of $\overline{\pi(V^*(\mathbf{f}(c)))}$. Taking into account the complexity estimates in Theorem \ref{thm:K-Proj}, this proves Theorem \ref{thm:Q-Proj}.

\bigskip

Finally, we show an example where, following the steps of the algorithm $\Q$-\texttt{Projection},
we obtain a geometric resolution of $\overline{\pi(V^*(\mathbf{f}))}$ for a sparse system $\mathbf{f}$
with rational coefficients.

\bigskip

\noindent \textbf{Example.}  Let $\mathbf{f}$ be the following sparse system with the same support family $\cA=(\cA_1, \cA_2)$ as in the example of Section \ref{sec:example}:
$$\mathbf{f}=\begin{cases}
f_1= 3+2X_1X_2X_3- X_1^2 X_4^4 X_5^2+ 5 X_4^8 X_5^4\\
f_2= 2 X_1 X_3 X_4 X_5^2-3 X_2 X_3^2 X_4^5 X_5^4 + 7 X_1 X_2^3 X_4
^5 X_5^4
\end{cases}$$
and let $\pi: \C^5 \to \C^3$ be the projection $\pi(x_1,x_2,x_3,x_4,x_5)=(x_1,x_2,x_3)$.

We use the same choices $b=1$, $\lambda=X_5$ and $\mu=X_3$
as in Section \ref{sec:example} and look at steps 4 and 6 of the algorithm.

In Step 4, the algorithm \texttt{ParametricToricGeomRes}
computes the geometric resolution of $V^*(\mathbf{f}_1)$ with free variables $X_1,X_2$ associated to the linear form  $\lambda$:

\begin{itemize}
\item $\widehat{q}_{X_5}(Y)= Y^{10}-\frac{2X_1^2}{5}Y^8+\frac{X_1^4+30}
{25}Y^6+\frac{2X_1^2}{75}Y^4-\frac{4X_1^4 -27+28X_1^3X_2^4}
{75}Y^2+\frac{4X_1^2}{25}$
\item $\widehat{w}_3(Y)= \frac{-5}{2X_1X_2} Y^4 + \frac{X_1}{2X_2} Y^2 -
\frac{3}{2 X_1 X_2}$

\item $\widehat{w}_5(Y)=Y.$
\end{itemize}

In Step 6, if Algorithm \texttt{GeomResProj} is applied to the geometric resolution obtained in Step 4 and the linear form $\mu$, the geometric resolution
$(\widehat{q}_{X_3},\widehat{v}_3)$ is obtained, where
\begin{itemize}
\item $\widehat{q}_{X_3}(Y)=
Y^5+\frac{3}{2X_1X_2}Y^4-\frac{14X_1X_2^4+ X_1^2}{3X_2^2}Y^3+\frac{-63X_2^4+ 10X_1}{9X_2^3}Y^2
+\frac{49X_1^2X_2^4+7X_1^3}{9}Y+ \frac{49X_1X_2^3}{6}$
\item $\widehat{v}_3(Y)=Y.$
\end{itemize}

This is, in fact, the geometric resolution of $ \overline{\pi(V^*(\mathbf{f}))} $ with free variables $X_1, X_2$ associated to the linear form $\mu=X_3$, as it can be checked, for instance, by applying a Groebner basis elimination based procedure.

\bigskip
\noindent
\textbf{Acknowledgements.} The authors thank the referees for their helpful comments and suggestions.

\end{document}